\documentclass[12pt]{article}
\usepackage{latexsym,amsfonts,amsmath,graphics,amssymb}
\usepackage{epsfig}

\newtheorem{theorem}{Theorem}
\newtheorem{lemma}{Lemma}
\newtheorem{corollary}{Corollary}
\newtheorem{proposition}{Proposition}

\newtheorem{remark}{Remark}

\newenvironment{proof}{\begin{trivlist}
    \item[\hskip\labelsep{\it Proof.}]}{$\hfill\Box$\end{trivlist}}

\newcommand{\satop}[2]{\stackrel{\scriptstyle{#1}}{\scriptstyle{#2}}}

\allowdisplaybreaks

\begin{document}

\title{Numerical integration of H\"older continuous, absolutely convergent  Fourier-, Fourier cosine-, and Walsh series}

\author{Josef Dick\footnote{School of Mathematics and Statistics, The University of New South Wales, Sydney, 2052 NSW, Australia; email: josef.dick@unsw.edu.au}}

\date{}
\maketitle

\begin{abstract}
We introduce quasi-Monte Carlo rules for the numerical integration of functions $f$ defined on $[0,1]^s$, $s \ge 1$, which satisfy the following properties: the Fourier-, Fourier cosine- or Walsh coefficients of $f$ are absolutely summable and $f$ satisfies a H\"older condition of order $\alpha$, for some $0 < \alpha \le 1$. We show a convergence rate of the integration error of order $\max((s-1) N^{-1/2}, s^{\alpha/2} N^{-\alpha} )$. The construction of the quadrature points is explicit and is based on Weil sums.
\end{abstract}


{\bf MSC Class:} 65D30, 65D32, 65C05, 65C10

\section{Introduction}

We study numerical approximation of high dimensional integrals $\int_{[0,1]^s} f(\boldsymbol{x}) \,\mathrm{d} \boldsymbol{x}$ by means of equal weight quadrature rules $\frac{1}{N} \sum_{n=0}^{N-1} f(\boldsymbol{x}_n)$ called quasi-Monte Carlo rules. A number of constructions of good quadrature points have been introduced, see \cite{DKS13, DP10, niesiam} for an overview. For these constructions many bounds on the integration error for various cases have been established, showing optimal convergence rates of the form $C_s (\log N)^{p(s)} N^{-1}$, where $p(s)$ grows linearly with the dimension, or of the form $C_{s,\delta} N^{-\delta}$, for all $0 < \delta < 1$, where $C_{s,\delta} > 0$ goes to $\infty$ as $\delta$ approaches $0$ and where $C_{s, \delta}$ grows exponentially fast with the dimension $s$. Since these bounds yield an exponential dependence on the dimension $s$, they are not informative if $N$ is much smaller than, say, $2^s$. To circumvent this situation, Sloan and Wo\'zniakowski \cite{SW98} 
introduced so-called weighted function spaces to remove the 
exponential dependence on the dimension and obtain meaningful upper bounds on the integration error even when $N$ is much smaller than $2^s$. The idea is to only consider 
functions for which the importance of variables decreases as the index increases. Furthermore, for the spaces considered in \cite{SW98}, upper bounds which do not depend on the dimension, or do so only polynomially, can only be obtained if such weights are introduced. The topic of tractability in its various forms studies the dependence on the dimension in detail, see the monographs of Novak and Wo\'zniakowski \cite{NW08, NW10, NW12}. 

Another problem in QMC theory is that integrands arising in applications are often not smooth, so that classical bounds on the integration error do not apply. For instance, integrands arising from financial applications are sometimes of the form $\max(0, \Phi(\boldsymbol{x}))$, which do not have bounded variation or finite norm in the function spaces considered in classical theory (see \cite{GKS13} where this issue is treated with some finance problems in mind).

To widen the scope of applications of QMC rules, it is therefore interesting to study function classes which are not weighted as in, say, \cite{SW98}, and which include non-smooth functions. The aim of this paper is to address numerical integration of some particular function classes which are not weighted and which contain non-smooth functions. We introduce explicit constructions of high dimensional quadrature point sets which work well in QMC rules in the range $s^2 < N < 2^s$. We prove an error bound of order $\max((s-1) N^{-1/2}, s^{\alpha/2} N^{-\alpha})$. The requirements on the integrand are: The Fourier coefficients, the Fourier cosine coefficients, or the Walsh coefficients of the integrand are absolutely summable and that the integrand satisfies a H\"{o}lder condition of order $0 < \alpha \le 1$.

In practice, situations where $N$ is much smaller than $2^s$ occur for various reasons. For instance, if the dimension $s$ is very large (say, $s=360$ as in \cite{PT}) then one can not use $2^s$ points (note $2^{360} \approx 10^{108}$) so that $N$ is much smaller than $2^s$. Another example is from \cite{KSS} where an infinite dimensional integration problem comes up which is then truncated. In this case the truncation error depends on the dimension, so that one has to balance the integration error with the truncation error. This yields a situation where $N \approx s^r$ for some small value of $r > 0$.  

Another result which does not require weighted function spaces when $N$ is much smaller than $2^s$ concerns the so-called star-discrepancy $D^\ast(P_{N,s})$ for an $N$-element point set $P_{N,s} \subset [0,1]^s$. The star-discrepancy is a measure for the distribution properties of a point set. It appears in the upper bound on the integration error for functions $f$ of bounded variation $V(f) < \infty$ in the sense of Hardy and Krause 
\begin{equation*}
\left|\int_{[0,1]^s} f(\boldsymbol{x}) \,\mathrm{d} \boldsymbol{x} - \frac{1}{N} \sum_{\boldsymbol{x} \in P_{N,s} } f(\boldsymbol{x}) \right| \le D^\ast(P_{N,s}) V(f).
\end{equation*}
This is the Koksma-Hlawka inequality, see \cite{Hlawka, Koksma} or also \cite{DP10, KN, niesiam}. It has been shown in \cite{HNWW} that there exist point sets in $[0,1]^s$ whose star-discrepancy is of order $\sqrt{s N^{-1}}$. Thus, even an integral in $100$ or more dimensions can reasonably be approximated using such a point set. The caveat though is, that no explicit constructions of such point sets are known. The existence proofs rely on concentration inequalities (like Hoeffdings inequality). Furthermore the scope of computer search algorithms for finding good point sets, see \cite{DGW10}, is limited since the computation of the star-discrepancy is NP-hard \cite{GSW09}.

From a practical point of view high dimensional numerical integration problems remain difficult, especially if there is no clear weighting of different coordinate directions and the functions are not smooth. In the following we describe some QMC rules which achieve a rate of order $\max((s-1) N^{-1/2}, s^{\alpha/2} N^{-\alpha} )$ for function classes with absolutely converging series expansions which satisfy a H\"older condition of order $0 < \alpha \le 1$.

\section{The method}

Our method is based on a bound from A. Weil~\cite{Weil} on exponential sums.  A proof of this result can for instance be found in  \cite[Theorem~5.36 on p. 220, Theorem~5.37 on p. 222 and Theorem~5.38 on p. 223]{LN97}, where it is referred to as `Weil sum'. See also \cite{CU57}, where some of the conditions have been relaxed.

To state this result, we write $\mathbb{F}_{b^m}$ for the finite field with $b^m$ elements, where $b$ is a prime and $m \ge 1$. An additive character of $\mathbb{F}_{b^m}$ is a function $\psi: \mathbb{F}_{b^m} \to T = \{z \in \mathbb{C}: |z|=1\}$ with the property that $\psi(x+y) = \psi(x) \psi(y)$ for all $x, y \in \mathbb{F}_{b^m}$. This implies that $\psi(0) = 1$. Further, there is a character which is given by $\psi(x) = 1$ for all $x \in \mathbb{F}_{b^m}$. This character is called the trivial character or principal character. The set of all polynomials in $z$ with coefficients in $\mathbb{F}_{b^m}$ is denoted by $\mathbb{F}_{b^m}[z]$. We state a version of Weil's result in the following.
\begin{proposition}[Weil sum]\label{WeilA}
Let $\mathbb{F}_{b^m}$ be a finite field with $b^m$ elements, where $b$ is prime and $m \ge 1$. Let $f \in \mathbb{F}_{b^m}[z]$ be a nonzero polynomial of degree $s \ge 1$ and let $\psi$ be a non-trivial additive character of $\mathbb{F}_{b^m}$. Then, provided that
\begin{equation}\label{cond_sN}
s < b^m \quad \mbox{and} \quad \mathrm{gcd}(s, b^m) = 1,
\end{equation}
we have
\begin{equation}\label{Weil_bound}
\left|\sum_{z \in \mathbb{F}_{b^m} } \psi(f(z)) \right| \le (s-1) \sqrt{b^m}.
\end{equation}
\end{proposition}
Since $|\psi(z)| = 1$ for all $z \in \mathbb{F}_{b^m}$ we have the trivial bound
\begin{equation*}
\left|\sum_{z \in \mathbb{F}_{b^m} } \psi(f(z)) \right| \le b^m.
\end{equation*}
Thus if $(s-1) \sqrt{b^m} \ge b^m$ the bound \eqref{Weil_bound} only yields a trivial result. For this reason we only consider the case where $s \le \sqrt{b^m}$ in this paper.

We consider now a special case of Proposition~\ref{WeilA}. Let $m=1$ and $N = b$, where $b$ is a prime number.  The finite field $\mathbb{F}_N$ is then just the finite field $\mathbb{Z}_N = \{0, 1, \ldots, N-1\}$ of residues modulo $N$, where we identify the elements in $\mathbb{Z}_N$ with the corresponding integers. The additive characters of $\mathbb{Z}_N$ are given by $\psi_a(x) = \mathrm{e}^{2\pi \mathrm{i} a x/N}$ for $a \in \mathbb{Z}_{N}$. For instance, $\psi_0$ is the trivial character. In the following we write $\mathbb{Z}$ for the set of integers, $\mathbb{N} = \{1, 2, \ldots\}$ for the set of natural numbers, and $\mathbb{N}_0$ for the set of nonnegative integers. For vectors $\boldsymbol{k} = (k_1, k_2, \ldots, k_s) \in \mathbb{Z}^s$ we write $N | \boldsymbol{k}$ if $N$ divides all components of $\boldsymbol{k}$, that is $N | k_j$ for $1 \le j \le s$. We write $N \not\,\mid \boldsymbol{k}$ if $N$ does not divide the vector $\boldsymbol{k}$.

\begin{proposition}\label{WeilB}
Let $N$ be a prime number and let $s \ge 1$ be a natural number such that $s \le \sqrt{N}$. Then for all $\boldsymbol{k} = (k_1, k_2, \ldots, k_s) \in \mathbb{Z}^s$ such that $N \not\,\mid \boldsymbol{k}$ we have
\begin{equation*}
\left|\sum_{n=0}^{N-1} \mathrm{e}^{2\pi \mathrm{i} (k_1 n + k_2 n^2 + \cdots + k_s n^s)/N} \right| \le (s-1) \sqrt{N}.
\end{equation*}
\end{proposition}

To cover all possible choices of $\boldsymbol{k} = (k_1, k_2, \ldots, k_s)$, note that we obviously have for $N | \boldsymbol{k}$ 
\begin{equation*}
\sum_{n=0}^{N-1} \mathrm{e}^{2\pi \mathrm{i} (k_1 n + k_2 n^2 + \cdots + k_s n^s)/N} = N.
\end{equation*}

We now relate Proposition~\ref{WeilB} to numerical integration of Fourier series. Proposition~\ref{WeilB} suggests to use the following quadrature points. Let $N$ be a prime number. Then we define the quadrature points
\begin{equation}\label{constr_Fourier}
\boldsymbol{x}_n = \left(\left\{\frac{n}{N} \right\}, \left\{\frac{n^{2} }{N} \right\}, \ldots, \left\{\frac{n^{s} }{N} \right\} \right) \quad \mbox{for } 0 \le n < N,
\end{equation}
where $\{x\} = x-\lfloor x \rfloor$ denotes the fractional part for nonnegative real numbers $x$. Let
\begin{equation}\label{PNs}
P_{N,s} = \{\boldsymbol{x}_0, \boldsymbol{x}_1, \ldots, \boldsymbol{x}_{N-1}\}.
\end{equation}

Note that $P_{N,s}$ given by \eqref{PNs} is similar to the construction studied in \cite{EHN94}. In fact, one can obtain similar results for the point sets from \cite{EHN94} as for \eqref{PNs}.

\begin{remark}
A straightforward implementation constructs the point set $P_{N,s}$ in order $Ns$ operations. However, the points can be generated in order $N$ operations using order $N$ storage in the following way. Let $N$ be a prime number and $g$ be a primitive element modulo $N$. Generate and store the vector $(a_0, a_1, \ldots, a_{N-2} )$ where $a_n = g^{n} \pmod{N}$. Then let $\boldsymbol{x}'_0 = (0, 0, \ldots, 0)$ and for $0 \le n < N-1$ let
\begin{equation}\label{constr_Fourier_fast}
\boldsymbol{x}'_{n+1} = \left(\frac{a_{n \pmod{N-1} }}{N}, \frac{a_{2 n \pmod{N-1}}}{N}, \ldots, \frac{a_{s n \pmod{N-1}}}{N}\right).
\end{equation}
Then $P_{N,s} = \{\boldsymbol{x}'_0, \boldsymbol{x}'_1, \ldots, \boldsymbol{x}'_{N-1} \}$ (only the ordering of the points is different).
\end{remark}

We denote the standard inner product by $\boldsymbol{k} \cdot \boldsymbol{x} = k_1 x_1 + k_2 x_2 + \cdots + k_s x_s$.  Proposition~\ref{WeilB} implies that for any $\boldsymbol{k} = (k_1, k_2, \ldots, k_s) \in \mathbb{Z}^s$ with $N \not\,\mid \boldsymbol{k}$ we have
\begin{equation}\label{trig_sum1}
\left| \frac{1}{N} \sum_{n=0}^{N-1} \mathrm{e}^{2\pi \mathrm{i} \boldsymbol{k} \cdot \boldsymbol{x}_n } \right| \le \frac{s-1}{\sqrt{N}}.
\end{equation}
Again, for $N \mid \boldsymbol{k}$ we have
\begin{equation}\label{trig_sum2}
\frac{1}{N} \sum_{n=0}^{N-1} \mathrm{e}^{2\pi \mathrm{i} \boldsymbol{k} \cdot \boldsymbol{x}_n }  = 1.
\end{equation}

We apply \eqref{trig_sum1} and \eqref{trig_sum2} to numerical integration of absolutely convergent Fourier series 
\begin{equation}\label{Fourier_series_F}
F(\boldsymbol{x}) = \sum_{\boldsymbol{k} \in \mathbb{Z}^s} a_{\boldsymbol{k}} \mathrm{e}^{2\pi \mathrm{i} \boldsymbol{k} \cdot \boldsymbol{x} }, 
\end{equation}
i.e., where for $a_{\boldsymbol{k}} \in \mathbb{C}$ we assume that $\sum_{\boldsymbol{k} \in \mathbb{Z}} |a_{\boldsymbol{k}}| < \infty$. Then we have
\begin{equation}\label{eq_int_Fourier}
\int_{[0,1]^s} F(\boldsymbol{x}) \,\mathrm{d} \boldsymbol{x} - \frac{1}{N} \sum_{n=0}^{N-1} F(\boldsymbol{x}_n) = - \sum_{\boldsymbol{k} \in \mathbb{Z}^s \setminus\{ \boldsymbol{0} \} } a_{\boldsymbol{k}} \frac{1}{N} \sum_{n=0}^{N-1} \mathrm{e}^{2\pi \mathrm{i} \boldsymbol{k} \cdot \boldsymbol{x}_n}.
\end{equation} 
In Section~\ref{sec_Fourier} we show that to use Equation~\eqref{trig_sum1}, we require that the Fourier series is absolutely convergent, and to obtain a convergence rate using Equation~\eqref{trig_sum2} it is sufficient to assume that the Fourier series satisfies a H\"older condition.

In the following we review the literature in which the point set $P_{N,s}$ given by \eqref{PNs} has been studied.

\section{Literature review}

The point set $P_{N,s}$ given by \eqref{PNs} and similar constructions have been studied before by Korobov~\cite{Korobov} in 1963 and Hua and Wang~\cite{HW} in 1978. However, at those times it was not customary to prove bounds which depend only polynomially on the dimension and so these classical results have an exponential dependence on the dimension. Since the dependence of the upper bound on the number of points $N$ is also weaker than for other point sets, the point set $P_{N,s}$ has not received much attention since those early results. We review the results by Korobov and Hua and Wang in the following. A summary of these results can also be found in \cite[Section~3.15.5]{SP05}.

Numerical integration of Fourier series was studied by Korobov. We describe \cite[Theorem 6, p. 72]{Korobov}. Let $F$ be the Fourier series \eqref{Fourier_series_F}. Assume that the Fourier coefficients satisfy $|a_{\boldsymbol{k}}| \le c \prod_{j=1}^s \min \{1, |k_j|^{-\alpha}\}$ for some $\alpha > 1$.  In \cite[Theorem 6, p. 72]{Korobov} it is shown that the integration error of Fourier series $F$ having Fourier coefficients $a_{\boldsymbol{k}}$ is bounded above by
\begin{equation*}
\left|\int_{[0,1]^s} F(\boldsymbol{x}) \,\mathrm{d} \boldsymbol{x} - \frac{1}{N} \sum_{\boldsymbol{x} \in P_{N,s}} F(\boldsymbol{x}) \right| \le \frac{(s-1)}{\sqrt{N}} \sum_{\boldsymbol{k} \in \mathbb{Z}^s} |a_{\boldsymbol{k}}| + \frac{c}{N^\alpha} \left(1 + 2 \sum_{k=1}^\infty k^{-\alpha} \right)^s.
\end{equation*}

Hua and Wang~\cite[p. 79]{HW} called the point set $P_{N,s}$ a `$p$-set'. From \cite[Theorem~4.9]{HW} we know that the star-discrepancy of this set is bounded above by $$D^*(P_{N,s}) \le c(s) N^{-1/2} (\ln N)^s$$ for all $N$ large enough. The constant $c(s)$ is given in \cite[p. 82]{HW}, which decreases exponentially fast with the dimension; however the upper bound still increases exponentially fast with the dimension for $N > 4$ due to the factor $(\ln N)^s$.

Hua and Wang~\cite[Theorem~7.3, p. 134]{HW} also prove an upper bound on the integration error for a different space of periodic functions with smoothness parameter $0 < \alpha \le 1$. This space is described in \cite[Chapter~6]{HW}. For $1/2 < \alpha \le 1$ the upper bound on the integration error is of the form $c(\alpha,s) N^{-\alpha} (\ln N)^{s-1} + c'(\alpha,s) N^{-1/2}$  and for $0 < \alpha \le 1/2$ it is of the form $c''(\alpha, s) N^{-\alpha} (\ln N)^{s-1+\delta(\alpha,1/2)}$, where $\delta(\alpha,1/2) = 1$ if $\alpha = 1/2$ and $0$ otherwise. 

In all these results the upper bounds grow exponentially fast as the dimension $s$ increases. In this paper we prove upper bounds which only show a polynomial growth with the dimension $s$. We also extend our results to non-periodic function classes. In the next setting we introduce the initial and worst-case error and polynomial tractability. The idea of the latter concept is to study the dependence on the dimension, see for instance \cite[Chapter~1]{NW08} for more background on these ideas.

\section{Worst-case error, initial error and polynomial tractability}


In the following we study the worst-case integration error for certain Banach spaces $B$ with norm $\|\cdot \|_B$ using quasi-Monte Carlo (QMC) rules based on point sets $U \subset [0,1]^s$ consisting of $N$ points. We define the worst-case error of integration in $B$ using a QMC rule based on the point set $U$ by
\begin{equation*}
\mathrm{e}(B, U) = \sup_{\satop{f \in B}{\|f\|_{B} \le 1}} \left|\int_{[0,1]^s} f(\boldsymbol{x}) \,\mathrm{d} \boldsymbol{x} - \frac{1}{|U|} \sum_{\boldsymbol{x} \in U} f(\boldsymbol{x}) \right|.
\end{equation*}
For reference, we also define the initial error as the integration error when approximating the integral by $0$, that is,
\begin{equation*}
\mathrm{e}(B, \emptyset) = \sup_{\satop{f \in B}{\|f\|_{B} \le 1}} \left|\int_{[0,1]^s} f(\boldsymbol{x}) \,\mathrm{d} \boldsymbol{x}  \right|.
\end{equation*}
For the spaces considered in this paper we always have $\mathrm{e}(B, \emptyset) = 1$.

For numerical integration in $B$, the QMC information complexity $N(\varepsilon, s)$ is the smallest number $N$ such that there exists a point set $U \subset [0,1]^s$ with $|U| = N$ and $e(B, U) \le \varepsilon e(B, \emptyset)$, that is
\begin{equation*}
N(\varepsilon, s) = \inf \{N \in \mathbb{N}: \exists U \subset [0,1]^s: |U| = N \mbox{ and } e(B, U) \le \varepsilon e(B, \emptyset)\}.
\end{equation*}
The general concept of information complexity includes the use arbitrary quadrature rules, not just QMC rules, see \cite[p. 1]{NW08}. The QMC information complexity is thus an upper bound on the information complexity.

The QMC information complexity shows how difficult it is to reduce the initial error by a factor of $\varepsilon$ using a QMC rule. Classically one is interested in studying the dependence of $N(\varepsilon, s)$ on $\varepsilon$. However, if $s$ is large, then also the dependence of $N(\varepsilon, s)$ on $s$ is often significant. To classify problems according to their difficulty, we introduce the concept of QMC polynomial tractability.

The approximation of integrals of functions in $B$ is \emph{polynomially QMC tractable}, if there exist real numbers $C, q > 0$ and $r \ge 0$ such that
\begin{equation}\label{N_eps_s_bound}
N(\varepsilon, s) \le C \varepsilon^{-q} s^r \quad \mbox{for all } \varepsilon \in (0,1) \mbox{ and } s \in \mathbb{N}.
\end{equation}
If $r=0$, then the integration problem is \emph{strongly polynomially QMC tractable}.

Replacing the QMC information complexity in \eqref{N_eps_s_bound} by the information complexity yields the concepts of \emph{polynomial tractability} and \emph{strong polynomial tractability}. Since the QMC information complexity is an upper bound on the information complexity, it follows that (strong) polynomial QMC tractability implies (strong) polynomial tractability. A thorough discussion of the general concepts can be found for instance in \cite[Chapter~1 and 2]{NW08}.

In the next section we discuss numerical integration of Fourier series using the point set $P_{N,s}$ given by \eqref{PNs}.

\section{Absolutely convergent Fourier series}\label{sec_Fourier}

\subsection{Function space}

We introduce a space of functions which can be represented by Fourier series (which we explain below). Let $f \in L^2([0,1]^s)$ and assume that $f$ is one-periodic in each variable, that is, $f(\boldsymbol{x}_u, \boldsymbol{0}) = f(\boldsymbol{x}_u, \boldsymbol{1})$ for all $u \subseteq \{1,2,\ldots, s\}$ and all $\boldsymbol{x}_u \in [0,1]^u$, where $\boldsymbol{x}_u \in [0,1]^u$ denotes a vector whose components are indexed by the elements $j \in u$, i.e., $\boldsymbol{x}_u = (x_j)_{j \in u}$ and $(\boldsymbol{x}_u, \boldsymbol{0})$ is the $s$-dimensional vector whose $j$th coordinate is $x_j$ for $j \in u$ and $0$ otherwise. Similarly, $\boldsymbol{x}_u, \boldsymbol{1})$ is the $s$-dimensional vector whose $j$th coordinate is $x_j$ for $j \in u$ and $1$ otherwise. Define the Fourier coefficient of $f$ by
\begin{equation*}
\widetilde{f}(\boldsymbol{k}) = \int_{[0,1]^s} f(\boldsymbol{x}) \mathrm{e}^{-2\pi \mathrm{i} \boldsymbol{k} \cdot \boldsymbol{x}} \,\mathrm{d} \boldsymbol{x}.
\end{equation*}
Then we have
\begin{equation*}
f(\boldsymbol{x}) \sim \sum_{\boldsymbol{k} \in \mathbb{Z}^s} \widetilde{f}(\boldsymbol{k}) \mathrm{e}^{2\pi \mathrm{i} \boldsymbol{k} \cdot \boldsymbol{x}} =: F(\boldsymbol{x}).
\end{equation*}

For $0 < \alpha \le 1$ and $1 \le p \le \infty$ we define the H\"older semi-norm
\begin{equation*}
|f|_{H_{\alpha,p}} = \sup_{\boldsymbol{x}, \boldsymbol{h}, \boldsymbol{x}+\boldsymbol{h} \in [0,1)^s} \frac{|f(\boldsymbol{x} + \boldsymbol{h})  - f(\boldsymbol{x}) |}{\|\boldsymbol{h} \|_{\ell^p}^\alpha},
\end{equation*}
where $\|\cdot \|_{\ell_p}$ denotes the $\ell^p$ norm and we define the norm
\begin{equation}\label{Fourier_norm}
\|f\|_{K_{\alpha,p}} = \sum_{\boldsymbol{k} \in \mathbb{Z}^s} \left|\widetilde{f}(\boldsymbol{k})\right| + |f|_{H_{\alpha,p}}.
\end{equation}
The norms $\|\cdot\|_{\ell^p}$ are of course all equivalent, however, the choice of $p$ will influence the dependence on the dimension of the worst-case error upper bound.

Note that we have
\begin{equation*}
\left( \int_{[0,1]^s} |f(\boldsymbol{x})|^2 \,\mathrm{d} \boldsymbol{x} \right)^{1/2} = \left( \sum_{\boldsymbol{k} \in \mathbb{Z}^s} |\widetilde{f}(\boldsymbol{k})|^2 \right)^{1/2} \le \sum_{\boldsymbol{k} \in \mathbb{Z}^s} |\widetilde{f}(\boldsymbol{k})| \le \|f\|_{K_{\alpha, p}}.
\end{equation*}
For $0 < \alpha \le 1$ and $1 \le p \le \infty$ we define the space $K_{\alpha,p}$ of functions $f: [0,1)^s \to \mathbb{R}$ as the set of all one-periodic functions $f \in L^2([0,1]^s)$ with finite norm, that is,
\begin{equation*}
K_{\alpha,p} = \{f \in L^2([0,1]^s): f \mbox{ is one-periodic and } \|f\|_{K_{\alpha,p}} < \infty\}.
\end{equation*}

Our main interest lies in the range $0 < \alpha \le 1/2$. If $0 < \alpha \le 1/2$, then it is known that there are periodic functions $f$ with $|f|_{H_{\alpha,p}} < \infty$ and $\sum_{\boldsymbol{k} \in \mathbb{Z}^s} \left| \widetilde{f}(\boldsymbol{k}) \right| = \infty$, see \cite[p. 35]{Ka04}. In the opposite direction, for instance, the Weierstra{\ss} function $w_\beta(x) = \sum_{k=0}^\infty 2^{-\beta k} \cos(2\pi 2^k x)$ provides an example of a function whose Fourier coefficients are absolutely summable for any $\beta > 0$, but $|w_\beta|_{H_{\alpha, p}} = \infty$ for every $\beta < \alpha \le 1$. Thus, both parts of the norm \eqref{Fourier_norm} provide independent conditions on $f$. 

We note that for any $f \in K_{\alpha,p}$, the Fourier series $F$ of $f$ converges to $f$ at every point in $[0,1]^s$. This follows directly from \cite[Corollary~1.8, p. 249]{StWe71}, using that $f$ is continuous since it satisfies a H\"older condition, i.e. $|f|_{H_{\alpha, p}} < \infty$.

\subsection{Numerical integration}

In this subsection we use \eqref{trig_sum1} and \eqref{trig_sum2} to obtain a bound on the integration error \eqref{eq_int_Fourier}. We split the right-hand side of  \eqref{eq_int_Fourier} into two parts. To obtain an estimation we need the following lemma which is most likely known (we include a proof for completeness).
\begin{lemma}\label{lem_Fourier_Holder}
For any $0 < \alpha \le 1$, $1 \le p \le \infty$, $f \in K_{\alpha, p}$ and $L \in \mathbb{N}$ we have
\begin{equation*}
\left| \sum_{\boldsymbol{k} \in \mathbb{Z}^s \setminus \{\boldsymbol{0}\} } \widetilde{f}(L \boldsymbol{k}) \right| \le \frac{s^{\alpha / p}}{L^\alpha} |f|_{H_{\alpha, p}}.
\end{equation*}
\end{lemma}

\begin{proof}
Using the Fourier series expansion of $f$, we obtain for any $L \in \mathbb{N}$ that
\begin{align*}
 \frac{1}{L^s} \sum_{\boldsymbol{\ell} \in \{0,1, \ldots, L-1\}^s} f\left(\frac{\boldsymbol{\ell}}{L} \right) = & \sum_{\boldsymbol{k} \in \mathbb{Z}^s} \widetilde{f}(\boldsymbol{k}) \frac{1}{L^s} \sum_{\boldsymbol{\ell} \in \{0, 1, \ldots, L-1\}^s} \mathrm{e}^{2\pi \mathrm{i} \boldsymbol{k} \cdot \boldsymbol{\ell} / L} = \sum_{\boldsymbol{k} \in \mathbb{Z}^s} \widetilde{f}(L \boldsymbol{k}),
\end{align*}
where the last inequality follows since
\begin{equation*}
\frac{1}{L^s} \sum_{\boldsymbol{\ell} \in \{0, 1, \ldots, L-1\}^s} \mathrm{e}^{2\pi \mathrm{i} \boldsymbol{k} \cdot \boldsymbol{\ell} / L} = \prod_{j=1}^s \frac{1}{L} \sum_{\ell_j=0}^{L-1} \mathrm{e}^{2\pi \mathrm{i} k_j \ell_j/L} = \left\{\begin{array}{rl}  1 & \mbox{if } L | \boldsymbol{k}, \\ 0 & \mbox{otherwise}.  \end{array} \right.
\end{equation*}
Thus we obtain
\begin{align*}
\left| \sum_{\boldsymbol{k} \in \mathbb{Z}^s \setminus \{\boldsymbol{0}\} } \widetilde{f}(L \boldsymbol{k}) \right| = & \left|\frac{1}{L^s} \sum_{\boldsymbol{k} \in \{0,1, \ldots, L-1\}^s} f\left(\frac{\boldsymbol{k}}{L} \right) - \int_{[0,1]^s} f(\boldsymbol{x}) \,\mathrm{d} \boldsymbol{x} \right| \\ = & \left|\frac{1}{L^s} \sum_{\boldsymbol{k} \in \{0,1,\ldots, L-1\}^s} \left(f\left(\frac{\boldsymbol{k}}{L} \right) - L^s \int_{[\boldsymbol{k} L^{-1}, (\boldsymbol{k} + \boldsymbol{1}) L^{-1})} f(\boldsymbol{x}) \,\mathrm{d} \boldsymbol{x} \right) \right|,
\end{align*}
where $[\boldsymbol{k} L^{-1}, (\boldsymbol{k} + \boldsymbol{1}) L^{-1}) = \prod_{j=1}^s [k_j L^{-1}, (k_j+1) L^{-1})$. Since $f$ is continuous, the mean value theorem implies that for each $\boldsymbol{k} \in \{0,1,\ldots, L-1\}^s$ there is a $\boldsymbol{y}_{\boldsymbol{k}} \in [\boldsymbol{k} L^{-1}, (\boldsymbol{k} + \boldsymbol{1}) L^{-1} )$ such that $L^s \int_{[\boldsymbol{k} L^{-1}, (\boldsymbol{k} + \boldsymbol{1}) L^{-1})} f(\boldsymbol{x}) \,\mathrm{d} \boldsymbol{x}  =  f(\boldsymbol{y}_{\boldsymbol{k}})$. Since  $\|\boldsymbol{y}_{\boldsymbol{k}} - \boldsymbol{k}/L \|_{\ell^p}$ is bounded by $s^{1/p} /L$, we obtain from the assumption that $f$ satisfies a H\"older condition $|f|_{H_{\alpha, p}} < \infty$ that
\begin{align*}
\left| \sum_{\boldsymbol{k} \in \mathbb{Z}^s \setminus \{\boldsymbol{0}\} } \widetilde{f}(L \boldsymbol{k}) \right| = & \left|\frac{1}{L^s} \sum_{\boldsymbol{k} \in \{0,1,\ldots, L-1\}^s} \left(f\left(\frac{\boldsymbol{k}}{L} \right) - f(\boldsymbol{y}_{\boldsymbol{k}}) \right) \right| \\ \le & \frac{1}{L^s} \sum_{\boldsymbol{k} \in \{0,1,\ldots, L-1\}^s} \left| f\left(\frac{\boldsymbol{k}}{L} \right) - f(\boldsymbol{y}_{\boldsymbol{k}}) \right| \\ \le & \frac{s^{ \alpha / p}}{L^\alpha} |f|_{H_{\alpha, p}},
\end{align*}
which implies the result.
\end{proof}

We now obtain a bound on the integration error for the function space $K_{\alpha, p}$.

\begin{theorem}\label{thm_per}
Let $0 < \alpha \le 1$ and $1 \le p \le \infty$. Then for any prime number $N$ and natural number $1 \le s < N$ we have
\begin{equation*}
\mathrm{e}(K_{\alpha, p}, P_{N,s}) \le \max\left(\frac{s-1}{\sqrt{N}},  \frac{s^{ \alpha / p }}{N^{\alpha}} \right).
\end{equation*}
\end{theorem}

\begin{proof}
Using the fact that $\sum_{\boldsymbol{k} \in \mathbb{Z}^s} \left|\widetilde{f}(\boldsymbol{k})\right| < \infty$, we have
\begin{align*}
& \left| \frac{1}{N} \sum_{n=0}^{N-1} f(\boldsymbol{x}_n) - \int_{[0,1]^s} f(\boldsymbol{x}) \,\mathrm{d} \boldsymbol{x} \right| \\ = & \left| \sum_{\boldsymbol{k} \in \mathbb{Z}^s \setminus \{\boldsymbol{0}\}} \widetilde{f}(\boldsymbol{k}) \frac{1}{N} \sum_{n=0}^{N-1} \mathrm{e}^{2\pi \mathrm{i} \boldsymbol{k} \cdot \boldsymbol{x}_n}  \right| \\ \le &  \left| \sum_{\satop{\boldsymbol{k} \in \mathbb{Z}^s \setminus \{\boldsymbol{0}\}}{N \not\,\,\mid \boldsymbol{k}}} \widetilde{f}(\boldsymbol{k})  \frac{1}{N} \sum_{n=0}^{N-1} \mathrm{e}^{2\pi \mathrm{i} \boldsymbol{k} \cdot \boldsymbol{x}_n}  \right| +  \left| \sum_{\satop{\boldsymbol{k} \in \mathbb{Z}^s \setminus \{\boldsymbol{0}\}}{N \mid \boldsymbol{k}}} \widetilde{f}(\boldsymbol{k})  \frac{1}{N} \sum_{n=0}^{N-1} \mathrm{e}^{2\pi \mathrm{i} \boldsymbol{k} \cdot \boldsymbol{x}_n}  \right|\\ \le &  \sum_{\satop{\boldsymbol{k} \in \mathbb{Z}^s \setminus \{\boldsymbol{0}\}}{N \not\,\,\mid \boldsymbol{k}}} \left|\widetilde{f}(\boldsymbol{k}) \right|\left|\frac{1}
{N} \sum_{n=0}^{N-1} \mathrm{e}^{2\pi \mathrm{i} \boldsymbol{k} \cdot \boldsymbol{x}_n}  \right| +  \left| \sum_{\boldsymbol{k} \in \mathbb{Z}^s \setminus \{\boldsymbol{0}\}} \widetilde{f}(N \boldsymbol{k}) \right| \\ \le & \frac{s-1}{\sqrt{N} } \sum_{\boldsymbol{k} \in \mathbb{Z}^s} \left|\widetilde{f}(\boldsymbol{k}) \right| + \frac{s^{ \alpha / p}}{N^\alpha} |f|_{H_{\alpha,p}} \\ \le  &  \max\left(\frac{s-1}{\sqrt{N}},  \frac{s^{ \alpha / p}}{N^{\alpha}} \right) \|f\|_{K_{\alpha,p}}.
\end{align*}
Thus the result follows.
\end{proof}

An immediate consequence of Theorem~\ref{thm_per} is the following result.
\begin{corollary}
Numerical integration in $K_{\alpha, p}$ is polynomially tractable.
\end{corollary}

\begin{remark}[Improved one dimensional projections]\label{rem_one_dim}
In the construction \eqref{constr_Fourier}, or equivalently \eqref{constr_Fourier_fast}, it can happen that some of the projections have less than $N$ different values, since $n^{N-1} \equiv 1 \pmod{N}$ for all $n \in \mathbb{Z}_N$, see for instance  \cite[p. 48, Lemma~2.3]{LN97}. As is well known, if $j \mid (N-1)$, then $n^{j} \pmod{N}$ can take on only $(N-1)/j$ different values. To see this, let $g$ be a primitive element, then $g^{j v} \pmod{N}$ only goes through the values $g^{j v}\pmod{N}$ for $0 \le v < (N-1)/j$. For the range $(N-1)/j \le v < N-1$ the values $g^{j v} \pmod{N}$ just get repeated.

To avoid this situation, we first generalize the definition of the quadrature points
\begin{equation*}
\boldsymbol{x}_{n, \boldsymbol{j}} = \left(\left\{\frac{n^{j_1}}{N} \right\}, \left\{\frac{n^{j_2} }{N} \right\}, \ldots, \left\{\frac{n^{j_s} }{N} \right\} \right) \quad \mbox{for } 0 \le n < N,
\end{equation*}
where $\boldsymbol{j} = (j_1, j_2,\ldots, j_s)$ with $1 \le j_1 < j_2 < \cdots < j_s < N$. Let
\begin{equation*}
P_{N,s, \boldsymbol{j} } = \{\boldsymbol{x}_{0, \boldsymbol{j}}, \boldsymbol{x}_{1, \boldsymbol{j}}, \ldots, \boldsymbol{x}_{N-1, \boldsymbol{j} }\}.
\end{equation*}
In the following we explain how to choose $\boldsymbol{j}$ to improve the one-dimensional projections. Let
\begin{equation*}
J = \{a \in \{1,2, \ldots, N-2\} : \mathrm{gcd}(a, N-1) = 1\} = \{a_1, a_2, a_3, \ldots, a_{|J|}\},
\end{equation*}
where $a_1 < a_2 < \cdots < a_{|J|}$. We choose now $j_i = a_i$ for $1 \le i \le s$. Let $\phi$ be Euler's totient function, then from \cite[Theorem~8.8.7]{BS96} we have for all $N > 3$ that
\begin{align*}
|J| = \phi(N-1) > \frac{N-1}{\mathrm{e}^\gamma \log \log (N-1) + \frac{3}{\log \log (N-1)} },
\end{align*}
where $\gamma = 0.5772\ldots$ is Euler's constant. Therefore, we have $|J| > \sqrt{N}$ for all prime numbers $N$. Thus for any $s \le \sqrt{N}$ we can choose $j_1, j_2, \ldots, j_s \in J$.

Theorem~\ref{thm_per} is almost the same when we replace $P_{N,s}$ by $P_{N,s, \boldsymbol{j}}$. One needs to replace $s$ in Proposition~\ref{WeilB} by $j_s$, which implies the error bound
\begin{equation*}
\mathrm{e}(K_{\alpha, p}, P_{N,s, \boldsymbol{j} }) \le \max\left(\frac{j_s-1}{\sqrt{N}},  \frac{s^{ \alpha / p}}{N^{\alpha}} \right).
\end{equation*}
See also \cite{B05}, where bounds on Weil sums with polynomials of the form $f(x) = a_1 x^{j_1} + a_2 x^{j_2} + \cdots + a_s x^{j_s}$ have been obtained under some restrictions.
\end{remark}

Another construction for which a bound similar to Weil's bound of Proposition~\ref{WeilB} holds, can be found in \cite{EHN94,NW11}, see \cite[Theorem~1]{EHN94} and also \cite[Theorem~1]{NW11} (where one chooses $k=1$ in the latter case). A bound similar to Theorem~\ref{thm_per} (and Theorem~\ref{thm_cos} using the tent transform below) can also be obtained for these constructions.

\section{Absolutely convergent Fourier cosine series}\label{sec_Cosine}

\subsection{Function space}

In the previous section we required that the functions are periodic. In this section we remove this condition by considering Fourier cosine series instead of Fourier series.

The cosine functions $\cos (k \pi x)$, $k \in \mathbb{N}_0$, form a complete orthogonal basis of $L^2([0,1])$. To normalize these functions we define
\begin{equation*}
\sigma_k(x) = \left\{\begin{array}{rl} 1 & \mbox{if } k = 0, \\ \sqrt{2} \cos(\pi k x) & \mbox{if } k \in \mathbb{N}. \end{array} \right.
\end{equation*}
Then we have
\begin{equation*}
\int_0^1 \sigma_k(x) \sigma_\ell(x) \,\mathrm{d} x = \left\{\begin{array}{rl} 1 & \mbox{if } k = \ell, \\ 0 & \mbox{if } k \neq \ell. \end{array} \right.
\end{equation*}
For vectors $\boldsymbol{k} \in \mathbb{N}_0^s$ and $\boldsymbol{x} \in [0,1]^s$ we define
\begin{equation*}
\sigma_{\boldsymbol{k}}(\boldsymbol{x}) = \prod_{j=1}^s \sigma_{k_j}(x_j).
\end{equation*}

For $g \in L^2([0,1]^s)$  we define the Fourier cosine coefficient
\begin{equation*}
\widehat{g}(\boldsymbol{k}) = \int_{[0,1]^s} g(\boldsymbol{x}) \sigma_{\boldsymbol{k}}(\boldsymbol{x}) \,\mathrm{d} \boldsymbol{x}
\end{equation*}
and the associated Fourier cosine series
\begin{equation*}
g(\boldsymbol{x}) \sim \sum_{\boldsymbol{k} \in \mathbb{N}_0^s} \widehat{g}(\boldsymbol{k}) \sigma_{\boldsymbol{k}}(\boldsymbol{x}) =: G(\boldsymbol{x}).
\end{equation*}

In the following we show how the theory developed for Fourier series can also be applied to Fourier cosine series by using the tent transform. This has previously been studied in the context of lattice rules in \cite{DNP13, H02}.

\subsubsection*{Tent transform}

The tent transform $\phi:[0,1] \to [0,1]$ is given by
\begin{equation*}
\phi(t) = 1 - |2t-1| = \left\{\begin{array}{rl} 2t & \mbox{if } 0 \le t < 1/2, \\ 2-2t & \mbox{if } 1/2 \le t \le 1. \end{array} \right.
\end{equation*}
For $\boldsymbol{t} = (t_1, t_2,\ldots, t_s) \in [0,1]^s$ we write $\phi(\boldsymbol{t}) = (\phi(t_1), \phi(t_2), \ldots, \phi(t_s))$.

The main properties of the tent transform which we use are the following: For any $k \in \mathbb{Z}$ and $0 \le t \le 1$ we have
\begin{equation}\label{eq_tent_cos}
\cos( \pi k \phi(t)) = \cos (2\pi k t)
\end{equation}
and for any $g \in L^2([0,1]^s)$ we have
\begin{equation*}
\int_{[0,1]^s} g(\boldsymbol{x}) \,\mathrm{d} \boldsymbol{x} = \int_{[0,1]^s} g(\phi(\boldsymbol{x})) \,\mathrm{d} \boldsymbol{x}.
\end{equation*}

Let $g \in L^2([0,1]^s)$, then the function $f(\boldsymbol{x}) = g(\phi(\boldsymbol{x}))$ is one-periodic. We have the following relation between the Fourier cosine coefficients of $g$ and the Fourier coefficients of $f$
\begin{align*}
\widehat{g}(\boldsymbol{k}_u, \boldsymbol{0}) = & \int_{[0,1]^s} g(\boldsymbol{x}) \sigma_{(\boldsymbol{k}_u, \boldsymbol{0})}(\boldsymbol{x}) \,\mathrm{d} \boldsymbol{x} \\ = & \int_{[0,1]^s} g(\phi(\boldsymbol{x})) \sigma_{(\boldsymbol{k}_u, \boldsymbol{0})}(\phi(\boldsymbol{x})) \,\mathrm{d} \boldsymbol{x} \\ = & 2^{-|u|/2} \int_{[0,1]^s} f(\boldsymbol{x}) \prod_{j \in u} (\mathrm{e}^{2\pi \mathrm{i} k_j x_j} + \mathrm{e}^{- 2\pi \mathrm{i} k_j x_j}) \,\mathrm{d} \boldsymbol{x} \\ = & 2^{-|u|/2} \sum_{v \subseteq u} \widetilde{f}(((-1)^{1_{j \in v}} k_j)_{j \in u}, \boldsymbol{0}),
\end{align*}
where $(((-1)^{1_{j \in v}} k_j)_{j \in u}, \boldsymbol{0})$ is the vector whose $j$th component is $-k_j$ for $j \in v$, $k_j$ for $j \in u \setminus v$ and $0$ otherwise. Thus we have
\begin{align*}
\sum_{\boldsymbol{k} \in \mathbb{Z}^s} \left|\widetilde{f}(\boldsymbol{k}) \right| = & \sum_{u \subseteq \{1,2,\ldots, s\}} \sum_{\boldsymbol{k}_u \in (\mathbb{Z} \setminus \{0\})^u} \left|\widetilde{f}((\boldsymbol{k}_u, \boldsymbol{0})) \right| \\ = & \sum_{u \subseteq \{1,2,\ldots, s\}} 2^{|u|/2} \sum_{\boldsymbol{k}_u \in \mathbb{N}^u} \left|\widehat{g}((\boldsymbol{k}_u, \boldsymbol{0})) \right|.
\end{align*}
The H\"older condition on $f$ is related to the H\"older condition on $g$ by
\begin{equation*}
|f|_{H_{\alpha,p} } = 2^\alpha |g|_{H_{\alpha,p} }.
\end{equation*}
It is therefore natural to define the following norm for (non-periodic) functions $g \in L^2([0,1]^s)$,
\begin{equation*}
\|g\|_{C_{\alpha, p}} = \sum_{u \subseteq \{1,2,\ldots, s\}} 2^{|u|/2} \sum_{\boldsymbol{k}_u \in \mathbb{N}^u} \left|\widehat{g}((\boldsymbol{k}_u, \boldsymbol{0})) \right| + 2^\alpha |g|_{H_{\alpha, p}}
\end{equation*}
and we denote the set of all functions $g \in L^2([0,1]^s)$ with finite norm $\|g\|_{C_{\alpha,p}} <\infty$ by $C_{\alpha, p}$.

By the above calculations, for $f(\boldsymbol{x}) = g(\phi(\boldsymbol{x}))$, we have the relation
\begin{equation*}
\|f\|_{K_{\alpha, p}} = \|g\|_{C_{\alpha, p}}.
\end{equation*}

From $f(\boldsymbol{x}) = g(\phi(\boldsymbol{x}))$ and the pointwise convergent of the Fourier series $F$ to the function $f$, we get
\begin{equation*}
g(\phi(\boldsymbol{x})) = f(\boldsymbol{x}) = F(\boldsymbol{x}) = G(\phi(\boldsymbol{x})) \quad \mbox{for all } \boldsymbol{x} \in [0,1]^s,
\end{equation*}
where the last equality follows from \eqref{eq_tent_cos} and from the fact that $F$ and $G$ are absolutely convergent. Thus we get that $G$ converges pointwise to $g$ for all points in $[0,1]^s$.

We note that Remark~\ref{rem_one_dim} applies also when applying the tent transformation for Fourier cosine series.

\subsection{Numerical integration}

We consider now the integration error using the following point set: Let $N$ be a prime number and let
\begin{equation*}
Q_{N,s} = \left\{ \boldsymbol{y}_n = \phi\left(\left\{\frac{n}{N}\right\}, \left\{\frac{n^2}{N} \right\}, \ldots, \left\{\frac{n^s}{N} \right\} \right), \quad \mbox{for } 0 \le n < N \right\},
\end{equation*}
where $\{z\} = z-\lfloor z \rfloor$ denotes the fractional part of $z \ge 0$. Note that $\boldsymbol{y}_n = \phi(\boldsymbol{x}_n)$ for $0 \le n < N$.

\begin{theorem}\label{thm_cos}
Let $0 < \alpha \le 1$ and $1 \le p \le \infty$. Let $N$ be a prime number and let $1 \le s < N$. Then we have
\begin{equation*}
\mathrm{e}(C_{\alpha, p} , Q_{N,s})  \le \max\left(\frac{s-1}{\sqrt{N}}, \frac{s^{ \alpha / p}}{N^\alpha} \right).
\end{equation*}
\end{theorem}

\begin{proof}
Let $g \in C_{\alpha, p} $ and set $f(\boldsymbol{x}) = g(\phi(\boldsymbol{x}))$. Then $f \in K_{\alpha, p}$ and $\|f\|_{K_{\alpha, p} } = \|g\|_{C_{\alpha, p}}$. Further we have
\begin{align*}
\int_{[0,1]^s} g(\boldsymbol{x}) \,\mathrm{d} \boldsymbol{x} - \frac{1}{N} \sum_{n=0}^{N-1} g(\boldsymbol{y}_n) = & \int_{[0,1]^s} f(\boldsymbol{x}) \,\mathrm{d} \boldsymbol{x} - \frac{1}{N} \sum_{n=0}^{N-1} f(\boldsymbol{x}_n).
\end{align*}
Thus we have
\begin{equation*}
e(C_{\alpha, p}, Q_{N,s}) = e(K_{\alpha, p}, P_{N,s})
\end{equation*}
and therefore the result follows from Theorem~\ref{thm_per}.
\end{proof}

An immediate consequence of Theorem~\ref{thm_cos} is the following result.
\begin{corollary}
Numerical integration in $C_{\alpha, p}$ is polynomially tractable.
\end{corollary}

\section{Absolutely convergent Walsh series}\label{sec_Walsh}

In this section we use Proposition~\ref{WeilA} with $m > 1$ instead of Proposition~\ref{WeilB} to construct quadrature rules. In this case we arrive at Walsh series instead of Fourier series.

Let $\psi$ be a fixed non-trivial additive character of $\mathbb{F}_{b^m}$. For $a \in \mathbb{F}_{b^m}$ define
\begin{equation*}
\theta_a(z) = \psi(a z).
\end{equation*}
The set $G^\ast = \{\theta_a: a \in \mathbb{F}_{b^m}\}$ defines all the additive characters of $\mathbb{F}_{b^m}$ and is a group induced by the additive group $(\mathbb{F}_{b^m}, +)$. Note that $G^\ast$ is a set of group homomorphisms from $(\mathbb{F}_{b^m}, +)$ to $T = \{z \in \mathbb{C}: |z|=1\}$, which is itself a group. This group $G^\ast$ is dual to $(\mathbb{F}_{b^m}, +)$. 

In the following we introduce some basic concepts which lead us to the definition of Walsh functions. We follow \cite{DM13} in our exposition, where more detailed information and references are provided. We define the direct product of denumerably many copies of $\mathbb{F}_{b^m}$ and the dual group $G^\ast$ by
\begin{align*}
\mathbb{X} = & \mathbb{F}_{b^m}^{\mathbb{N}} = \{\boldsymbol{x} = (\xi_1, \xi_2, \ldots): \xi_i \in \mathbb{F}_{b^m} \}, \\ \mathbb{K} = & \{\boldsymbol{k} = (\kappa_1, \kappa_2, \ldots) \in (G^\ast)^{\mathbb{N}}: \kappa_i = 0 \mbox{ for almost all } i\}.
\end{align*}
With this definition $\mathbb{X}$ and $\mathbb{K}$ are dual to each other via the pairing
\begin{equation*}
\mathbb{K} \times \mathbb{X} \to T, \quad (\kappa_1, \kappa_2, \ldots) \bullet (\xi_1, \xi_2, \ldots) = \prod_{i=1}^\infty \kappa_i(\xi_i) \in T.
\end{equation*}

Let $\varphi: \mathbb{Z}_{b^m} \to \mathbb{F}_{b^m}$ and $\phi: \mathbb{Z}_{b^m} \to G^\ast$ be bijections with $\varphi(0) = 0$ and $\phi(0) = 0$. We assume that $\phi$ is defined via $\phi(a) = \theta_{\mu(a)}$, where $\mu:\mathbb{Z}_{b^m} \to \mathbb{F}_{b^m}$ is another bijection with $\mu(0) = 0$. We extend these definitions (by an abuse of notation) to
\begin{align*}
\varphi:[0,1) \to & \mathbb{X}, \\ \varphi(x) = & (\varphi(\xi_1), \varphi(\xi_2), \ldots),
\end{align*}
where $x$ has $b^m$-adic expansion $x = \xi_1 b^{-m} + \xi_2 b^{-2m} + \cdots$ and where we use the finite expansion if $x$ is a $b^m$-adic rational. Further we set
\begin{align*}
\phi: \mathbb{N}_0 \to & \mathbb{K}, \\ \phi(k) = & (\phi(\kappa_0), \phi(\kappa_1), \ldots, \phi(\kappa_{a-1}), 0, 0, \ldots),
\end{align*}
where $k$ has $b^m$-adic expansion $k = \kappa_0 + \kappa_1 b^m + \cdots + \kappa_{a-1} b^{m (a-1)}$.

We now define the Walsh functions by
\begin{align*}
\mathrm{wal}: [0,1) \to & T, \\ \mathrm{wal}_k(x) = & \phi(k) \bullet \varphi(x) = \prod_{i=1}^a \phi(\kappa_{i-1})(\varphi(\xi_i)).
\end{align*}
We also define a higher dimensional analogue by
\begin{align*}
\mathrm{wal}: [0,1)^s \to & T, \\ \mathrm{wal}_{\boldsymbol{k}}(\boldsymbol{x}) = & \prod_{j=1}^s \mathrm{wal}_{k_j}(x_j).
\end{align*}
The Walsh functions are an orthonormal basis of $L^2([0,1]^s)$. For more background information on Walsh functions in the context of numerical integration see also \cite{LNS96}.

\subsection{Construction of point set}

We now describe how to construct point sets. In order to be able to use Proposition~\ref{WeilA}, we need to ensure that the polynomials $f \in \mathbb{F}_{b^m}[z]$ of degree $s$ satisfy $\gcd(s, b^m) = 1$. The last condition is equivalent to $b \not\!| s$, since $b$ is a prime number. In order to satisfy this condition for all admissible choices of the coefficients of $f$, we use the sequence of exponents $1, 2, \ldots, b-1, b+1, b+2,\ldots, 2b-1, 2b+1, \ldots$, i.e., we remove all multiples of $b$ from the natural numbers. By assumption, $b$ is a prime number, therefore every number in this sequence has greatest common divisor $1$ with $b$. Let $c_j \in \mathbb{N}$ be the $j$th element in this sequence, i.e,  $$c_j - \left \lfloor \frac{c_j}{b} \right\rfloor  =j.$$ Note that 
\begin{equation}\label{deg_poly}
c_j \le j \frac{b}{b-1}. 
\end{equation}

Let $\varphi: \mathbb{Z}_{b^m} \to \mathbb{F}_{b^m}$ be given as above and let $\varphi^{-1}: \mathbb{F}_{b^m} \to \mathbb{Z}_{b^m}$ denote the inverse mapping $\varphi^{-1}(\varphi(n)) = n$. Then for $n \in \mathbb{Z}_{b^m}$ we define
\begin{equation}\label{constr_point_dig}
\boldsymbol{z}_n = \left(\frac{ \varphi^{-1}[ (\varphi(n))^{c_1} ] }{b^m}, \frac{\varphi^{-1} [(\varphi(n))^{c_2} ] }{b^m}, \ldots, \frac{\varphi^{-1} [ (\varphi(n))^{c_s} ] }{b^m} \right) \in [0,1)^s,
\end{equation}
and we set
\begin{equation*}
R_{b^m, s} = \{\boldsymbol{z}_0, \boldsymbol{z}_1, \ldots, \boldsymbol{z}_{b^m-1}\}.
\end{equation*}

\subsubsection*{An elementary construction}

We now describe a special case of how to construct the points. Let $p \in \mathbb{Z}_b[x]$ be an irreducible polynomial of degree $m$. For a polynomial $q(x) = q_0 + q_1 x + \cdots + q_{m-1} x^{m-1} \in \mathbb{Z}_b[x]/ (p)$ we define the mapping
\begin{align*}
\nu_m: \mathbb{Z}_b[x]/(p) & \to [0,1) \\ \nu_m(q) = & \frac{q_0}{b} + \frac{q_1}{b^2}  + \cdots + \frac{q_{m-1}}{b^m}.
\end{align*}

Let $0 \le n < b^m$ be given by $n = \zeta_0 + \zeta_1 b + \cdots + \zeta_{m-1} b^{m-1}$ and let $n(x) = \zeta_0 + \zeta_1 x + \cdots + \zeta_{m-1} x^{m-1} \in \mathbb{Z}_b[x]/(p)$ denote the associated polynomial. Then we set
\begin{equation*}
z_{j,n} = \nu_m(n^{c_j}(x))
\end{equation*}
(where $n^{c_j}(x)$ is computed in $\mathbb{Z}_b[x]/(p)$, i.e., modulo $p$) and we define the points
\begin{equation*}
\boldsymbol{z}_n = (z_{1,n}, z_{2,n},\ldots, z_{s,n}) \quad \mbox{for } 0 \le n < b^m.
\end{equation*}
Let
\begin{equation*}
R_{b^m,s} = \{\boldsymbol{z}_0, \boldsymbol{z}_1, \ldots, \boldsymbol{z}_{b^m-1}\}.
\end{equation*}

A fast construction can again be obtained in the following way. Let $p \in \mathbb{Z}_b[x]$ be a primitive polynomial. Then the polynomial $g(x) =x$ is a primitive element in the multiplicative group $(\mathbb{Z}_b[x]/(p) ) \setminus \{0\}$, that is, $x^{n} \pmod{p}$ for $0 \le n < b^m-1$ generates all nonzero elements in $\mathbb{Z}_b[x]/(p)$. For $0 \le n < b^m-1$ let
\begin{equation*}
a_n(x) \equiv x^{n} \pmod{p}.
\end{equation*}
Then set $\boldsymbol{z}'_0 = (0, 0, \ldots, 0)$ and for $0 \le n < b^m-1$ set
\begin{equation*}
\boldsymbol{z}'_{n+1} = (\nu_m(a_n), \nu_m(a_{2n \pmod{b^m-1}}), \ldots, \nu_m(a_{sn \pmod{b^m-1}})).
\end{equation*}
Then we have
\begin{equation*}
R_{b^m,s} = \{\boldsymbol{z}'_0, \boldsymbol{z}'_1, \ldots, \boldsymbol{z}'_{b^m-1}\}.
\end{equation*}

\subsection{Weil's result in terms of Walsh functions}

We can now write Proposition~\ref{WeilA} in the following way. For $\boldsymbol{k} \in \mathbb{N}_0^s$ with $b^m \not\,\mid \boldsymbol{k}$, $\boldsymbol{k} = (k_1, k_2, \ldots, k_s)$ and $k_j$ having $b^m$-adic representation $k_j = \kappa_{j,0} + \kappa_{j,1} b^m + \kappa_{j,2} b^{2m} +  \cdots$ we have
\begin{align*}
\mathrm{wal}_{\boldsymbol{k}}(\boldsymbol{z}_n) = & \prod_{j=1}^s \phi(\kappa_{j,0})((\varphi(n))^{c_i}) \\ = & \psi\left(\mu(\kappa_{1,0}) (\varphi(n))^{c_1} + \mu(\kappa_{2,0}) (\varphi(n))^{c_2} + \cdots + \mu(\kappa_{s,0}) (\varphi(n))^{c_s} \right).
\end{align*}
Note that the choice of $c_1, c_2, \ldots, c_s$ ensures that the degree of the polynomial $\phi(\kappa_{1,0}) x^{c_1} + \phi(\kappa_{2,0}) x^{c_2} + \cdots + \phi(\kappa_{s,0}) x^{c_s}$ is not divisible by $b$. By the construction of the $c_1 < c_2 < \cdots < c_s$, the degree of this polynomial is now bounded by $s \frac{b}{b-1}$, see \eqref{deg_poly}. Proposition~\ref{WeilA} now implies the following result.

\begin{proposition}\label{WeilC}
Let $\mathbb{F}_{b^m}$ be a finite field with $b^m$ elements, where $b$ is prime and $m \ge 1$. Let $s \ge 1$ be an integer. Let $\boldsymbol{\ell} \in \mathbb{N}_0^s$ be such that $b^m \not\,\,\mid \boldsymbol{\ell}$ and let $\boldsymbol{z}_n$ be given by \eqref{constr_point_dig}. Then we have
\begin{equation}\label{Weil_boundC}
\left|\sum_{n=0}^{b^m-1} \mathrm{wal}_{\boldsymbol{\ell}}( \boldsymbol{z}_n ) \right| \le \left(s \frac{b}{b-1} -1 \right) \sqrt{b^m}.
\end{equation}
\end{proposition}

\subsection{Function space}

We now introduce a space of functions which can be represented by Walsh series. Let $h \in L^2([0,1]^s)$ and for $\boldsymbol{k} \in \mathbb{N}_0^s$ define the Walsh coefficient
\begin{equation*}
\widehat{h}(\boldsymbol{k}) = \int_{[0,1]^s} h(\boldsymbol{x}) \overline{\mathrm{wal}_{\boldsymbol{k}}(\boldsymbol{x}}) \,\mathrm{d} \boldsymbol{x}.
\end{equation*}
Then we have
\begin{equation*}
h(\boldsymbol{x})  \sim \sum_{\boldsymbol{k} \in \mathbb{N}_0^s} \widehat{h}(\boldsymbol{k}) \mathrm{wal}_{\boldsymbol{k}}(\boldsymbol{x}) := H(\boldsymbol{x}).
\end{equation*}
We define the norm
\begin{equation*}
\|h\|_{W_{\alpha,p}} = \sum_{\boldsymbol{k} \in \mathbb{N}_0^s} \left|\widehat{h}(\boldsymbol{k}) \right| + |h|_{H_{\alpha,p}}.
\end{equation*}
For $0 < \alpha \le 1$ and $1 \le p \le \infty$ we define the space $W_{\alpha,p}$ of functions $h:[0,1)^s \to \mathbb{R}$ as the set of all functions in $L^2([0,1]^s)$ with finite norm, that is,
\begin{equation*}
W_{\alpha,p} = \{h \in L^2([0,1]^s): \|h\|_{W_{\alpha, p}} < \infty\}.
\end{equation*}

For $h \in W_{\alpha,p}$ we have that $h$ is continuous and $\sum_{\boldsymbol{k} \in \mathbb{N}_0^s} \left|\widetilde{h}(\boldsymbol{k})\right| < \infty$, therefore, an analogous result as \cite[Theorem~A.20]{DP10} implies that $h(\boldsymbol{x}) = H(\boldsymbol{x})$ for all $\boldsymbol{x} \in [0,1)^s$.

\subsection{Numerical integration}

The following lemma is most likely known (for completeness we include a proof).
\begin{lemma}
For any $0 < \alpha \le 1$, $1 \le p \le \infty$, $h \in W_{\alpha, p}$ and $L \in \mathbb{N}$ with $b^m \mid L$, we have
\begin{equation*}
\left| \sum_{\boldsymbol{k} \in \mathbb{N}_0^s \setminus \{\boldsymbol{0}\} } \widetilde{h}(L \boldsymbol{k}) \right| \le \frac{s^{ \alpha / p}}{L^\alpha} |h|_{H_{\alpha, p}}.
\end{equation*}
\end{lemma}

\begin{proof}
For any $L \in \mathbb{N}$ with $b^m \mid L$ we have
\begin{align*}
 \frac{1}{L^s} \sum_{\boldsymbol{\ell} \in \{0,1, \ldots, L-1\}^s} h\left(\frac{\boldsymbol{\ell}}{L} \right) = & \sum_{\boldsymbol{k} \in \mathbb{N}_0^s} \widehat{h}(\boldsymbol{k}) \frac{1}{L^s} \sum_{\boldsymbol{\ell} \in \{0, 1, \ldots, L-1\}^s} \mathrm{wal}_{\boldsymbol{k}}( \boldsymbol{\ell} / L) = \sum_{\boldsymbol{k} \in \mathbb{N}_0^s} \widehat{h}(L \boldsymbol{k}),
\end{align*}
where the last inequality follows since
\begin{equation*}
\frac{1}{L^s} \sum_{\boldsymbol{\ell} \in \{0, 1, \ldots, L-1\}^s} \mathrm{wal}_{\boldsymbol{k} }( \boldsymbol{\ell} / L) = \prod_{j=1}^s \frac{1}{L} \sum_{\ell_j=0}^{L-1} \mathrm{wal}_{k_j}(\ell_j/L) = \left\{\begin{array}{rl}  1 & \mbox{if } L | \boldsymbol{k}, \\ 0 & \mbox{otherwise}.  \end{array} \right.
\end{equation*}
In the last step we need the condition that $b^m \mid L$. The remainder of the proof follows by the same arguments as the proof of Lemma~\ref{lem_Fourier_Holder}.
\end{proof}

We can now obtain a bound on the worst-case integration error for $W_{\alpha, p}$ using the point set $R_{N,s}$. The proof is essentially the same as the proof of Theorem~\ref{thm_per} and is therefore omitted.

\begin{theorem}\label{thm_wal}
Let $0 < \alpha \le 1$ and $1 \le p \le \infty$. Let $b$ be a prime number, $m \ge 1$ and $N = b^m$. Then for any natural number $1 \le s < N$ we have
\begin{equation*}
\mathrm{e}(W_{\alpha, p}, R_{N,s}) \le \max\left(\frac{b(s-1)+1}{(b-1) \sqrt{N}},  \frac{s^{ \alpha / p}}{N^{\alpha}} \right).
\end{equation*}
\end{theorem}

An immediate consequence of Theorem~\ref{thm_wal} is the following result.
\begin{corollary}
Numerical integration in $W_{\alpha, p}$ is polynomially tractable.
\end{corollary}

Remark~\ref{rem_one_dim} also applies for the rules constructed in this section. By choosing the exponents $j_i$ such that $\mathrm{gcd}(j_i, b^m-1) = 1$ for all $1 \le i \le s$, we obtain that the one-dimensional projections consist of the points $0, 1/b^m, 2/b^m, \ldots, (b^m-1)/b^m$.

\subsection*{Acknowledgment}

Josef Dick is supported by an ARC Queen Elizabeth II Fellowship. 

The author is grateful to the reviewer for pointing out the results by Korobov~\cite{Korobov} and Hua and Wang~\cite{HW}.

\end{document}